\documentclass{ws-sd}%
\usepackage{amsmath}%
\usepackage{amsfonts}%
\usepackage{amssymb}%
\usepackage{graphicx}
\begin{document}

\title{INVARIANT MEASURES FOR RANDOM EXPANDING ON AVERAGE SAUSSOL MAPS}
\author{Fawwaz Batayneh$^{*}$ and Cecilia Gonz\'{a}lez-Tokman$^{\dagger}$}
\maketitle

\begin{abstract}
In this paper, we investigate the existence of random absolutely continuous
invariant measures (ACIP) for random expanding on average Saussol maps in
higher dimensions. This is done by the establishment of a random Lasota-Yorke
inequality for the transfer operators on the space of bounded oscillation. We
prove that the number of ergodic skew product ACIPs is finite and provide an
upper bound for the number of these ergodic ACIPs. This work can be seen as a
generalization of the work in \cite{BGT} on admissible random Jab\l o\'{n}ski
maps to a more general class of higher dimensional random maps.

\end{abstract}

\markboth{Fawwaz Batayneh and Cecilia Gonz\'{a}lez-Tokman}{Invariant measures for random expanding on average Saussol maps}

\address{School of Mathematics and Physics, The University of Queensland,\\ Saint Lucia, Queensland
4072, Australia\\
$^{*}$f.batayneh@uq.edu.au\\$^{\dagger}$cecilia.gt@uq.edu.au}

\begin{history}
\received{(19 June 2021)}
\revised{(09 September 2021)}
\accepted{(15 September 2021)}
\end{history}


\section{Introduction}

In this paper, we study a class of random expanding on average
multi-dimensional maps where randomness means, at each iteration of the
discrete process, one of a given family of maps is chosen and applied to
produce the next stage of the dynamics. The randomness is governed by an
external ergodic, invertible, probability preserving system $\sigma
:\Omega\rightarrow\Omega$ (quenched setting) where $(\Omega,\mathbb{P})$ is a
probability space, including but not restricted to the IID case, where maps
are chosen according to a stationary process, see \cite{ANV}. In real life
applications, the relevance of random dynamical systems is clear due to the
fact that systems are influenced by external factors or noise.

The long term behaviour of these random maps is unpredictable. Therefore, one
attempts to understand their statistical properties through random absolutely
continuous invariant measures (ACIPs), which are relevant to describe
physically observable events. Generally, there is no measure that is invariant
under all the maps at the same time. In this paper, we investigate the
existence and bound the number of ergodic skew product ACIPs for the so-called
random Saussol maps.

The statistical properties of multi-dimensional piecewise expanding maps are
significantly more complicated to analyse than those of one dimensional maps.
This is not simply because of technical difficulties, indeed there are
intrinsic obstacles. For example, there exist two-dimensional piecewise
expanding and $C^{2}$ maps with singular ergodic properties \cite{TS3}, and
examples with no or infinitely many absolutely continuous invariant measures
\cite{noorinf}.

In \cite{S},\ Saussol studied the statistical properties and existence of
absolutely continuous invariant measures with respect to the Lebesgue measure
for a general class of multi-dimensional piecewise expanding maps with
singularities. The author establishes a spectral gap in the transfer operators
associated to Saussol maps, defined in \cite[Section 2]{S}. The assumptions on
these maps naturally appear for maps with discontinuities on some wild sets.
The author proved the existence of a finite number ACIPs. Moreover, a key
property of the function space involved in \cite{S} gives a constructive upper
bound on the number of ergodic ACIPs.

In the literature, Saussol maps were studied by several researchers, among
many, we mention \cite{ANV,HV,Thomine}. In \cite{HV}, Hu and Vaienti treated a
class of nonsingular transformations with indifferent fixed points without the
assumption of any Markov property. They adapted Saussol's strategy to prove a
Lasota-Yorke inequality and obtained the existence of ACIPs that can be finite
or infinite. The random IID case of Saussol maps was covered by Aimino, Nicol
and Vaienti in \cite{ANV}. Their work is restricted to the case where the
partitions associated to the maps under consideration are finite.
Consequently, they were able to use the sufficient condition given in Lemma
$2.1$ in \cite{S} instead of the hypothesis described in (PE5) given in
Section $2$ in \cite{S}. They assumed the so-called random covering property
to show that their systems are mixing. In \cite{Thomine}, Thomine linked his
work and proved similar results to Saussol \cite{S} but using Sobolev spaces.

In a more recent work, in \cite{DFCV}, Dragi\v{c}evi\'{c}, Froyland,
Gonz\'{a}lez-Tokman and Vaienti defined admissible transfer operator cocycles.
They presented two classes of examples, one and higher dimensional piecewise
expanding maps. In the higher dimensional case, Saussol maps were used over
finite partitions with uniform constants. The same authors, in \cite{DFCV2},
proved a fiberwise almost sure invariance principle for a large class of
random dynamical systems. They also provided explicit examples of random
dynamical systems, including Saussol maps, and proved the existence of a
unique random ACIP. In \cite{TPS}, Tanzi, Pereira and van Strien studied
random compositions of small perturbations of dynamical systems modeled by
Saussol maps. Their main focus was to study when\ the compositions of
perturbations of a given map result in statistical behaviour close to that of
the map itself. Particularly, they proved that the evolution of sufficiently
regular mass distributions under the random perturbations stays close to the
mass distribution that is invariant under the perturbed map.

This work can be seen as a generalization of the work in \cite{BGT} on random
Jab\l o\'{n}ski maps where each component of the map only depends on its
corresponding variable. In this paper, we include maps such that the
components are allowed to depend on all or some of the variables. In
\cite{BGT}, the authors studied the quenched setting of random Jab\l o\'{n}ski
maps. They proved that the skew product associated to this random dynamical
system admits a finite number of ergodic ACIPs. Moreover, two different upper
bounds on the number of ergodic skew product ACIPs were provided. Those bounds
were motivated by the works of Buzzi \cite{B} on random Lasota-Yorke maps in
one dimension and G\'{o}ra, Boyarsky and Proppe \cite{GBP} on higher
dimensional deter. These bounds used the fact that Jab\l o\'{n}ski maps are
defined on rectangular partitions and the maps preserve rectangles. Such
bounds are not direct to adapt in our setting since we do not make any
assumption on the shape of partition and we do not assume any type of
geometric preservation under the maps. The bound we develop in \eqref{skewbd}
is based on an analytical observation that is nonnegative functions in the
space of bounded oscillation are strictly positive on nontrivial balls inside
their support.

In Theorem $2$ in \cite{GBP}, G\'{o}ra, Boyarsky and Proppe obtained a bound
on the number of ergodic ACIPs for piecewise $C^{2}$ Jab\l o\'{n}ski
transformations which are sufficiently expansive. In Remark $2.3$ in
\cite{liv}, Liverani gave a sufficient condition for the uniqueness of ACIP
using dense orbits. However, in spite of the fact that Batayneh and
Gonz\'{a}lez-Tokman \cite{BGT}, Buzzi \cite{B} and Araujo--Solano \cite{AS}
provided bounds on the number of ACIPs for random compositions of certain
classes of one and multi-dimensional maps, studying and, particularly,
bounding the number of ergodic ACIPs is still largely open problem. This
relates to questions about multiplicity of Lyapunov exponents in
multiplicative ergodic theory.

The plan of the paper is the following: in Section $2$, we provide background
materials regarding the function space involved and Oseledets splittings for
random dynamical systems. Section $3$ is devoted for defining random Saussol
maps. In Section $4$, we develop a random Lasota-Yorke Inequality and prove
the quasi-compact property. Section $5$ provides existence results of random
invariant and skew product ACIPs. In Section $6$, we provide an upper bound on
the number of ergodic skew product ACIPs for random Saussol maps. An example
is also provided.

\section{Background}

\subsection{The space of bounded oscillation}

There are several function spaces to use when studying transfer operators
induced by higher dimensional expanding maps. One of them is the space of
functions of bounded variation in higher dimension, see \cite{GBMULTI}.
Another alternative is fractional Sobolev spaces, as done in \cite{Thomine}.
In this paper, the analysis of the transfer operators of random Saussol maps
requires us to use a function space called the space of bounded oscillation.
This space was first introduced by Keller \cite{KE1} in one dimension,
developed by Blank \cite{BL} and used by Saussol \cite{S} and successively by
Buzzi \cite{B2} and Tsujii \cite{TS1}. Other references where this space was
used to provide densities of ACIPs are \cite{BK,HV}. In the rest of the paper,
$N>1$.

\begin{definition}
For a Borel subset $C$ of $%
\mathbb{R}
^{N}$ and $f\in L^{1}(%
\mathbb{R}
^{N})$, we define \emph{the oscillation of }$f$\emph{\ on }$C$ by%
\begin{equation}
osc(f,C):=\underset{C}{E\sup f}-\underset{C}{E\inf f}\text{,} \label{osc}%
\end{equation}
where $\underset{C}{E\sup f}$ is the essential supremum of $f$ on $C$ and
$\underset{C}{E\inf f}$ is the essential infimum of $f$ on $C$.
\end{definition}

For $x\in%
\mathbb{R}
^{N}$ and $\varepsilon>0$, denote the open ball of radius $\varepsilon$
centered at $x\ $by$\ B_{\varepsilon}(x)$. The mapping $x\mapsto
osc(f,B_{\varepsilon}(x))$ is lower semi-continuous by Proposition $3.1$ in
\cite{S} and hence measurable, one then can define the $\alpha$-seminorm of
$f$ (or the $\alpha$-oscillation of $f$).

\begin{definition}
\label{def:norm} \label{le copy(2)}Let $0<\alpha\leq1$ and $\tilde
{\varepsilon}_{0}>0$ be real numbers and $f\in L^{1}(%
\mathbb{R}
^{N})$. \emph{The }$\alpha$\emph{-seminorm of }$f$\emph{\ (or the }$\alpha
$\emph{-oscillation of }$f$\emph{)} is defined as%
\begin{equation}
\left\vert f\right\vert _{\alpha}:=\underset{\varepsilon\leq\tilde
{\varepsilon}_{0}}{\sup}\varepsilon^{-\alpha}\underset{%
\mathbb{R}
^{N}}{\int}osc(f,B_{\varepsilon}(x))dx\text{.} \label{bdedosc}%
\end{equation}
\end{definition}

\begin{definition}
Let $f\in L^{1}(C)$. We identify $f$ with its extension by zero to $%
\mathbb{R}
^{N}$. If $\left\vert f\right\vert _{\alpha}$ is a finite, then $f$ is said to
be \emph{of bounded oscillation on }$C$. The set of all such maps is denoted
by $V_{\alpha}$. For $f\in V_{\alpha}$, the norm of $f$ is defined by
$\left\Vert f\right\Vert _{\alpha}:=\left\Vert f\right\Vert _{L^{1}%
}+\left\vert f\right\vert _{\alpha}$.
\end{definition}

The space $V_{\alpha}$ is called the space of bounded oscillation (the
Quasi-H\"{o}lder space in \cite{S}). $V_{\alpha}$ is a Banach space and
compactly embedded in $L^{1}(C)$, see \cite{S,B2}. Note that while the norm
depends on $\tilde{\varepsilon}_{0}$, the space $V_{\alpha}$ does not, and two
choices of $\tilde{\varepsilon}_{0}$ give rise to two equivalent norms.\ By
Lemma $1$ in \cite{BK}, the space $V_{1}$ is a proper subset of the space of
bounded variation $(BV(C),\Vert\cdot\Vert_{BV})$ in the sense of Definition
$1.1$ in \cite{G}, indeed $\left\Vert f\right\Vert _{BV}\leq2N^{\frac{3}{2}%
}\left\Vert f\right\Vert _{1}$. A crucial property of this space is, in higher
dimensions, functions of $V_{\alpha}$ are bounded \cite[Proposition $3.4$]{S}
but functions of $BV(C)$ are not bounded in general.

The oscillation defined in Equation \eqref{osc} satisfies some properties
listed in Proposition $3.2$ in \cite{S}. We use these properties when we
develop the Lasota-Yorke inequality in Theorem \ref{quco}. We recall this
proposition as well as Lemma $3.1$ from \cite{S} which gives one of the key
properties of nonnegative functions of the space of bounded oscillation
$V_{\alpha}$.

\begin{proposition}
[{\cite[Proposition $3.2$]{S}}]\label{properties}Let $f,f_{i},g\in L^{\infty}(%
\mathbb{R}
^{N})$, $g$ be positive function, $a,b,c>0$ and $K$ be Borel subset of $%
\mathbb{R}
^{N}$. The oscillation has the following properties:
(i) $osc(%
{\displaystyle\sum\limits_{i}}
f_{i},B_{a}(\cdot))\leq%
{\displaystyle\sum\limits_{i}}
osc(f_{i},B_{a}(\cdot))$.
(ii) $osc(f1_{K},B_{a}(\cdot))\leq osc(f,K\cap B_{a}(\cdot))1_{K}%
(\cdot)+2[\underset{K\cap B_{a}(\cdot)}{E\sup}\left\vert f\right\vert
]1_{B_{a}(K)\cap B_{a}(K^{c})}(\cdot)$, where $B_{a}(K):=\{x:$ $d(x,K)<a\}$
and $d$ is the Euclidean metric and $K^{c}$ is the complement of $K$.
(iii) $osc(fg,K)\leq osc(f,K)\underset{K}{E\sup}g+osc(g,K)\underset{K}{E\inf
}\left\vert f\right\vert $.
(iv) If $a+b\leq c$, then for all $x\in%
\mathbb{R}
^{N}$ we have
\[
\underset{B_{a}(x)}{E\sup}f\leq\frac{1}{m(B_{b}(x))}\underset{B_{b}(x)}{\int
}\Big(f(z)+osc(f,B_{c}(z))\Big)dz.
\]
\end{proposition}

\begin{lemma}
[{\cite[Lemma $3.1$]{S}}]\label{ph copy(1)}For every positive $h\in V_{\alpha
}$, $h\neq0$, there exists a ball on which the infimum of $h$ is strictly
positive. The radius $\varepsilon$ of the ball can be taken as%
\begin{equation}
\varepsilon=\min(\tilde{\varepsilon}_{0},\Big(\frac{\int hdm}{\left\vert
h\right\vert _{\alpha}}\Big)^{\frac{1}{\alpha}})\text{,}%
\end{equation}
where $\tilde{\varepsilon}_{0}$ as of Definition \eqref{le copy(2)}.
\end{lemma}

\subsection{Random dynamical systems and Oseledets splittings}

\begin{definition}
\label{RSM}A \emph{random dynamical system} is a tuple $\mathcal{R}%
=(\Omega,\mathcal{F},\mathbb{P},\sigma,\mathcal{X},\mathcal{L}\mathbb{)}$,
where the base $\sigma$ is an invertible measure-preserving transformation of
the probability space $(\Omega,\mathcal{F},\mathbb{P)}$, $(\mathcal{X}%
,\Vert\cdot\Vert)$ is a Banach space and $\mathcal{L}:\Omega
\mathcal{\rightarrow}L(\mathcal{X},\mathcal{X})$ is a family of bounded linear
maps of $\mathcal{X}$, called the generator.
\end{definition}

A key regularity notion in this work is the concept of $\mathbb{P}$-continuity
which was first introduced by Thieullen in \cite{Thieullen}. We will apply
this in Corollary \ref{coro1}.

\begin{definition}
For a topological space $\Omega$, equipped with a Borel probability
$\mathbb{P}$, a mapping $\mathcal{L}$ from $\Omega$ to a topological space $Y$
is said to be $\mathbb{P}$\emph{-continuous} if $\Omega$ can be expressed as a
countable union of Borel sets such that the restriction of $\mathcal{L}$ to
each of them is continuous.
\end{definition}

For convenience, we let $\mathcal{L}_{\omega}:=\mathcal{L}(\omega)$ be the
transfer operator defined in \eqref{ohdpoh}. A random dynamical system defines
a cocycle, given by
\begin{equation}
(k,\omega)\mapsto\mathcal{L}_{\omega}^{(k)}:=\mathcal{L}_{\sigma^{k-1}\omega
}\circ\dots\circ\mathcal{L}_{\sigma\omega}\circ\mathcal{L}_{\omega}\text{.}
\label{toc}%
\end{equation}

Multiplicative ergodic theorems deal with random dynamical systems
$\mathcal{R}=(\Omega,\mathcal{F},$ $\mathbb{P},\sigma,\mathcal{X}%
,\mathcal{L}\mathbb{)}$.
They give rise to Oseledets splittings or decompositions of $\mathcal{X}$ that
depend on $\omega$. We apply the Oseledets splitting theorem for $\mathbb{P}%
$-continuous random dynamical systems \cite{FLQ} to show that the random
invariant densities $h_{\omega}$ given in Theorem~\ref{muer} belong to the
first Oseledets subspace. In Corollary \ref{coro1}, we use the finite
dimensionality of the first Oseledets subspace to show that the number of
ergodic ACIPs is finite.

\begin{definition}
Let $\mathcal{R}=(\Omega,\mathcal{F},\mathbb{P},\sigma,\mathcal{X}%
,\mathcal{L}\mathbb{)}$ be a random dynamical system. \emph{An Oseledets
splitting for }$\mathcal{R}$ consists of a sequence of isolated (exceptional)
Lyapunov exponents
\[
\infty>\lambda^{\ast}=\lambda_{1}>\lambda_{2}>...>\lambda_{l}>\mathcal{K}%
^{\ast}\geq-\infty\text{,}%
\]
where the index $l$ $\geq1$ is allowed to be finite or countably infinite, and
a family of $\omega$-dependent splittings,%
\begin{equation}
\mathcal{X}=Y_{1}(\omega)\oplus...\oplus Y_{l}(\omega)\oplus V(\omega)\text{,}
\label{y1}%
\end{equation}
where for $j=1,...,l$, $d_{j}:=\dim(Y_{j}(\omega))<\infty$ and $V(\omega
)\in\mathcal{G(X)},$ where $\mathcal{G(X)}$ is the Grassmannian of
$\mathcal{X}$.
For all $j=1,...,l$ and $\mathbb{P}$-a.e. $\omega\in\Omega$ we have%
\begin{align*}
\mathcal{L}_{\omega}Y_{j}(\omega)  &  =Y_{j}(\sigma\omega)\text{,}\\
\mathcal{L}_{\omega}V(\omega)  &  \subseteq V(\sigma\omega)\text{,}%
\end{align*}
and
\begin{align}
\lim_{s\rightarrow\infty}\frac{1}{s}\log\left\Vert \mathcal{L}_{\omega}%
^{(s)}y\right\Vert  &  =\lambda_{j}\text{, }\forall y\in Y_{j}(\omega
)\backslash\{0\}\text{,}\label{keke}\\
\lim_{s\rightarrow\infty}\frac{1}{s}\log\left\Vert \mathcal{L}_{\omega}%
^{(s)}v\right\Vert  &  \leq\mathcal{K}^{\ast}\text{, }\forall v\in
V(\omega)\text{.}\nonumber
\end{align}
\end{definition}

\section{Random Saussol maps}

\begin{definition}
\label{le copy(1)} Let $(\Omega,\mathcal{F},\mathbb{P})$ be a probability
space and $\sigma:\Omega\circlearrowleft$ an invertible, ergodic and
$\mathbb{P-}$preserving transformation. Let $C$ be a compact subset of $%
\mathbb{R}
^{N}$, with $clos(int(C))\neq\phi$. A \emph{random Saussol map }%
$T$\emph{\ over }$\sigma$ is a map $T:\Omega\rightarrow\{T_{\omega}%
\}_{\omega\in\Omega}$, where $T_{\omega}:=T(\omega):C\circlearrowleft$ such
that there exists an at most countable family of disjoint open
sets \footnote{The sets $U_{i}$ and $V_{i}$ may also depend on $\omega$.
However, we do not make this dependence explicit, unless it becomes relevant for the discussion.} $U_i\subset C$ and $V_i$ where $clos(U_i)\subset V_i$
for all $i$ in an indexing set $\mathcal{I}$, and maps
\[
T_{\omega,i}:V_{i}\rightarrow%
\mathbb{R}
^{N}\text{,}%
\]
satisfying for some $0<\alpha\leq1$ and small enough $\varepsilon_0>0:$
\textbf{(PE1)} There exists $S>0,$ such that for $\mathbb{P}$-a.e. $\omega
\in\Omega$, $0<s(\omega)<S$ and for all $i\in\mathcal{I}$ and $u,v\in
T_{\omega}V_{i}$ such that $d(u,v)\leq\varepsilon_{0}$, we have
\begin{equation}
d(T_{\omega,i}^{-1}u,T_{\omega,i}^{-1}v)\leq s(\omega)d(u,v)\text{,}
\label{entaayah}%
\end{equation}
and
\begin{equation}
\int_{\Omega}\log(s(\omega))d\mathbb{P(\omega)<}0\text{.} \label{ec}%
\end{equation}
\textbf{(PE2)} For $\mathbb{P}$-a.e. $\omega\in\Omega$ and all $i$,
$T_{\omega}\mid_{U_{i}}=T_{\omega,i}\mid_{U_{i}}$ and $T_{\omega,i}%
(V_{i})\supset B_{\varepsilon_{0}}(T_{\omega,i}(U_{i}))$.
\textbf{(PE3)} For $\mathbb{P}$-a.e. $\omega\in\Omega$ and all $i$,
$T_{\omega,i}\in C^{1}(V_{i})$ and $T_{\omega,i}$ is injective and
$T_{\omega,i}^{-1}\in C^{1}(T_{\omega,i}V_{i})$. Moreover, the determinant is
uniformly H\"{o}lder: for $\mathbb{P}$-a.e. $\omega\in\Omega$ and all $i$,
$\varepsilon\leq\varepsilon_{0}$, $z\in T_{\omega,i}V_{i}$ and $x,y\in
B_{\varepsilon}(z)\cap T_{\omega,i}V_{i}$, we have%
\begin{equation}
\left\vert \det DT_{\omega,i}^{-1}x-\det DT_{\omega,i}^{-1}y\right\vert \leq
c\left\vert \det DT_{\omega,i}^{-1}z\right\vert \varepsilon^{\alpha}\text{,}
\label{detalphaholder}%
\end{equation}
for some $c>0$.
\textbf{(PE4)} For $\mathbb{P}$-a.e. $\omega\in\Omega$, $m(C\backslash
\underset{i\in\mathcal{I}}{\cup}U_{i})=0$.
\textbf{(PE5)} For $\mathbb{P}$-a.e. $\omega\in\Omega$ and $\varepsilon>0$,
let
\begin{equation}
G_{\omega,\varepsilon_{0}}(\varepsilon):=\underset{x\in C}{\sup}\text{
}G_{\omega,\varepsilon_{0}}(x,\varepsilon)\text{,} \label{G}%
\end{equation}
where%
\begin{equation}
G_{\omega,\varepsilon_{0}}(x,\varepsilon):=%
{\displaystyle\sum\limits_{i\in\mathcal{I}}}
\frac{m\Big(T_{\omega,i}^{-1}B_{\varepsilon}(\partial T_{\omega}U_{i})\cap
B_{(S-s(\omega))\varepsilon_{0}}(x)\Big)}{m(B_{(S-s(\omega))\varepsilon_{0}%
}(x))}\text{.} \label{gg}%
\end{equation}
For $\mathbb{P}$-a.e. $\omega\in\Omega$, define $\zeta_{\varepsilon_{0}%
}(\omega)$ by%
\begin{equation}
\zeta_{\varepsilon_{0}}(\omega):=s(\omega)^{\alpha}+2\underset{\varepsilon
\leq\varepsilon_{0}}{\sup}\frac{G_{\omega,\varepsilon_{0}}(\varepsilon
)}{\varepsilon^{\alpha}}(S\varepsilon_{0})^{\alpha}\text{,} \label{kekekek}%
\end{equation}
and assume that%
\begin{equation}
\int_{\Omega}\log(\zeta_{\varepsilon_{0}}(\omega))d\mathbb{P(\omega
)<}0\text{.} \label{iugb}%
\end{equation}
In addition, we assume the mapping $\omega\mapsto\mathcal{L}_{\omega}$ is
$\mathbb{P}$-continuous. For $k\in%
\mathbb{N}
$, the $k$-fold composition $T_{\omega}^{(k)}$ is defined as\
\begin{equation}
T_{\omega}^{(k)}=T_{\sigma^{k-1}\omega}\circ...\circ T_{\sigma\omega}\circ
T_{\omega}\text{.} \label{kokoko}%
\end{equation}
\end{definition}

For simplicity, we sometimes refer to the range of $T$, that is $\{T_{\omega
}\}_{\omega\in\Omega}$, as the random Saussol map. A random Saussol map gives
rise to a random dynamical system, where $\mathcal{X}=V_{\alpha}$ and
$\mathcal{L}_{\omega}=\mathcal{L}_{T_{\omega}}$ is the transfer operator
defined by
\begin{equation}
\mathcal{L}_{\omega}f=%
{\displaystyle\sum\limits_{i\in\mathcal{I}}}
(g_{\omega}f)\circ T_{\omega,i}{}^{-1}1_{T_{\omega}U_{i}}\text{,}
\label{ohdpoh}%
\end{equation}
where%
\[
g_{\omega}:=\frac{1}{\left\vert \det DT_{\omega}\right\vert }\text{.}%
\]

This can be seen as the expanding on average random version of the
deterministic maps studied by Saussol in \cite{S}.

In the above definition, we ensure that $C$ and the $U_{i}$'s do not need to
be connected and no control on the angles between smooth elements of the
partition like the ones in \cite{GBMULTI,C}. The family of the $U_{i}$'s does
not need to be finite.

\begin{remark}
For the random Saussol map $T=\{T_{\omega}\}_{\omega\in\Omega}$, the skew
product map $F$ on $\Omega\times C$ which encodes the whole dynamics of the
system is given by%
\begin{equation}
F(\omega,x)=(\sigma\omega,T_{\omega}(x))\text{. } \label{sp}%
\end{equation}
\end{remark}

\section{Lasota-Yorke inequality and quasi-compactness}

In this section, we establish a one step Lasota-Yorke inequality on the space
of bounded oscillation $V_{\alpha}$. This inequality is used to show the
quasi-compactness property in Corollary \ref{quasicom}.

\begin{definition}
Let $A:\mathcal{X}\circlearrowleft$ be a bounded linear map. \emph{The index
of compactness norm of }$A$ is
\[
\left\Vert A\right\Vert _{ic(\mathcal{X})}=\inf\{r>0:A(B_{\mathcal{X}})\text{
can be covered by finitely many balls of radius }r\}\text{,}%
\]
where $B_{\mathcal{X}}$ is the unit ball in $\mathcal{X}$.
\end{definition}

\begin{definition}
\label{le} Let $\mathcal{R}$ $=$ $(\Omega,\mathcal{F},\mathbb{P}%
,\sigma,\mathcal{X},\mathcal{L})$ be a random dynamical system such that
$\int_{\Omega}\log^{+}\left\Vert \mathcal{L}_{\omega}\right\Vert
d\mathbb{P(\omega)<\infty}.$ Let $\omega\in\Omega,$ \emph{the maximal Lyapunov
exponent }$\lambda(\omega)$ is%
\[
\lambda(\omega)=\lim_{k\rightarrow\infty}\frac{1}{k}\log\left\Vert
\mathcal{L}_{\omega}^{(k)}\right\Vert \text{,}%
\]
and \emph{the index of compactness }$\mathcal{K}(\omega)$ is%
\[
\mathcal{K}(\omega)=\lim_{k\rightarrow\infty}\frac{1}{k}\log\left\Vert
\mathcal{L}_{\omega}^{(k)}\right\Vert _{ic(\mathcal{X})}\text{,}%
\]
whenever these limits exist.
\end{definition}

We recall the following remark from \cite{CQ}.

\begin{remark}
\label{re4}In Definition \ref{le}, if $\sigma$ is ergodic, then $\lambda$ and
$\mathcal{K}$ are constant $\mathbb{P}-$almost everywhere. We denote these
constants by $\lambda^{\ast}$ and $\mathcal{K}^{\ast}$. By definition, we have
that $\mathcal{K}^{\ast}\leq\lambda^{\ast}$. The finiteness of $\lambda^{\ast
}$ is implied by the assumption that $\int_{\Omega}\log^{+}\left\Vert
\mathcal{L}_{\omega}\right\Vert d\mathbb{P(\omega)<}\infty$.
\end{remark}

\begin{definition}
A random dynamical system $\mathcal{R}$ with an ergodic base $\sigma$ is
called \emph{quasi-compact} if $\mathcal{K}^{\ast}<\lambda^{\ast}$.
\end{definition}

\begin{theorem}
\label{quco}Let $T=\{T_{\omega}\}_{\omega\in\Omega}$ be a random Saussol map.
If $\varepsilon_{0}$ in Definition~\ref{def:norm} is small enough, then there
are positive measurable functions $\eta,D:\Omega\rightarrow%
\mathbb{R}
^{+}$ such that $\int_{\Omega}\log(\eta(\omega))d\mathbb{P(\omega)<}0$ and
\[
\left\Vert \mathcal{L}_{\omega}f\right\Vert _{\alpha}\leq\eta(\omega
)\left\Vert f\right\Vert _{\alpha}+D(\omega)\left\Vert f\right\Vert _{L^{1}%
}\text{,}%
\]
for $\mathbb{P}$-a.e. $\omega\in\Omega$ and $f\in V_{\alpha}$.
\end{theorem}

\begin{proof}
Let $x=(x_{1},...,x_{n})\in C$, $\omega\in\Omega$, let $\alpha\in(0,1]$ and
$\varepsilon_{0}>0$ be small enough. By Proposition \eqref{ohdpoh} and
\ref{properties} $($i$)$, we have%
\[
osc(\mathcal{L}_{\omega}f,B_{\varepsilon}(x))\leq%
{\displaystyle\sum\limits_{i\in\mathcal{I}}}
osc\Big((g_{\omega}f)\circ(T_{\omega,i})^{-1}1_{T_{\omega}U_{i}}%
,B_{\varepsilon}(x)\Big)\text{.}%
\]
By Proposition \ref{properties} $($ii$)$, we get%
\begin{align}
osc(\mathcal{L}_{\omega}f,B_{\varepsilon}(x))  &  \leq%
{\displaystyle\sum\limits_{i\in\mathcal{I}}}
\left(
\begin{array}
[c]{c}%
osc((g_{\omega}f)\circ T_{\omega,i}{}^{-1},T_{\omega}U_{i}\cap B_{\varepsilon
}(x))1_{T_{\omega}U_{i}}(x)\\
+2[\underset{T_{\omega}U_{i}\cap B_{\varepsilon}(x)}{E\sup}\left\vert
(g_{\omega}f)\circ T_{\omega,i}{}^{-1}\right\vert ]1_{B_{\varepsilon}(\partial
T_{\omega}U_{i})}(x)
\end{array}
\right) \nonumber\\
&  \leq%
{\displaystyle\sum\limits_{i\in\mathcal{I}}}
\left(
\begin{array}
[c]{c}%
osc(g_{\omega}f,U_{i}\cap(T_{\omega})^{-1}B_{\varepsilon}(x))1_{T_{\omega
}U_{i}}(x)\\
+2[\underset{U_{i}\cap(T_{\omega,i})^{-1}B_{\varepsilon}(x)}{E\sup}\left\vert
g_{\omega}f\right\vert ]1_{B_{\varepsilon}(\partial T_{\omega}U_{i})}(x)
\end{array}
\right)  \text{.} \label{falafel}%
\end{align}
Let
\[
R_{\omega,i}^{(1)}(x):=osc\Big(g_{\omega}f,U_{i}\cap(T_{\omega})^{-1}%
B_{\varepsilon}(x)\Big)\text{.}%
\]
For $x\in T_{\omega}U_{i}$, let $y_{\omega,i}:=T_{\omega,i}{}^{-1}x$. Then, by
(PE1), we have%
\[
R_{\omega,i}^{(1)}(x)\leq osc\Big(g_{\omega}f,U_{i}\cap B_{s(\omega
)\varepsilon}(y_{\omega,i})\Big)\text{.}%
\]
By Proposition \ref{properties} $($iii$)$, for almost all $x\in T_{\omega
}U_{i}$, we have%
\[
R_{\omega,i}^{(1)}(x)\leq osc(f,B_{s(\omega)\varepsilon}(y_{\omega
,i}))\underset{U_{i}\cap B_{s(\omega)\varepsilon}(y_{\omega,i})}{E\sup
}g_{\omega}+osc(g_{\omega},B_{s(\omega)\varepsilon}(y_{\omega,i}%
))\underset{U_{i}\cap B_{s(\omega)\varepsilon}(y_{\omega,i})}{E\inf}\left\vert
f\right\vert \text{.}%
\]
Applying (PE3), we have%
\[
R_{\omega,i}^{(1)}(x)\leq(1+cs(\omega)^{\alpha}\varepsilon^{\alpha
})osc(f,B_{s(\omega)\varepsilon}(y_{\omega,i}))g_{\omega}(y_{\omega
,i})+\left\vert f\right\vert (y_{\omega,i})g_{\omega}(y_{\omega,i}%
)cs(\omega)^{\alpha}\varepsilon^{\alpha})\text{.}%
\]
Hence, the first term in \eqref{falafel} can be estimated as
\[%
{\displaystyle\sum\limits_{i\in\mathcal{I}}}
R_{\omega,i}^{(1)}1_{T_{\omega}U_{i}}\leq(1+cs(\omega)^{\alpha}\varepsilon
^{\alpha})\mathcal{L}_{\omega}(osc(f,B_{s(\omega)\varepsilon}(\cdot
)))+cs(\omega)^{\alpha}\varepsilon^{\alpha}\mathcal{L}_{\omega}(\left\vert
f\right\vert )\text{.}%
\]
Integrating both sides yields
\[
\underset{%
\mathbb{R}
^{N}}{\int}%
{\displaystyle\sum\limits_{i\in\mathcal{I}}}
R_{\omega,i}^{(1)}1_{T_{\omega}U_{i}}\leq(1+cs(\omega)^{\alpha}\varepsilon
^{\alpha})\underset{%
\mathbb{R}
^{N}}{\int}osc(f,B_{s(\omega)\varepsilon}(\cdot))+cs(\omega)^{\alpha
}\varepsilon^{\alpha}\underset{%
\mathbb{R}
^{N}}{\int}\left\vert f\right\vert \text{.}%
\]
By definition of $\mid f\mid_{\alpha}$ in \eqref{bdedosc} with $\tilde
{\varepsilon}_{0}=S\varepsilon_{0}$, we have
\begin{equation}
\underset{%
\mathbb{R}
^{N}}{\int}%
{\displaystyle\sum\limits_{i\in\mathcal{I}}}
R_{\omega,i}^{(1)}1_{T_{\omega}U_{i}}\leq(1+cs(\omega)^{\alpha}\varepsilon
^{\alpha})(s(\omega)\varepsilon)^{\alpha}\left\vert f\right\vert _{\alpha
}+c(s(\omega)\varepsilon)^{\alpha}\left\Vert f\right\Vert _{L^{1}}\text{.}
\label{kekeke}%
\end{equation}
For the second term in \eqref{falafel}, let
\[
R_{\omega,i}^{(2)}(x):=\Big(\underset{U_{i}\cap(T_{\omega,i})^{-1}%
B_{\varepsilon}(x)}{E\sup}\left\vert g_{\omega}f\right\vert
\Big)1_{B_{\varepsilon}(\partial T_{\omega}U_{i})}(x)\text{.}%
\]
If $x\notin B_{\varepsilon}(T_{\omega}U_{i})$ then $R_{\omega,i}^{(2)}(x)=0$.
Using the definition of $g_{\omega}$, (PE1) and \eqref{detalphaholder}, we get%
\[
R_{\omega,i}^{(2)}(x)\leq\Big(\underset{B_{\varepsilon}(y_{\omega,i})}{E\sup
}\left\vert f\right\vert \Big)\left\vert \det D(T_{\omega,i})^{-1}x\right\vert
(1+cs(\omega)^{\alpha}\varepsilon^{\alpha})1_{B_{\varepsilon}(\partial
T_{\omega}U_{i})}(x)\text{.}%
\]
Integrating both sides over $%
\mathbb{R}
^{N}$ followed by a change of variable $x=T_{\omega,i}y_{\omega,i}$ gives%
\begin{align}
&  \frac{1}{(1+cs(\omega)^{\alpha}\varepsilon^{\alpha})}\underset{%
\mathbb{R}
^{N}}{\int}R_{\omega,i}^{(2)}(x)dx\label{aa}\\
&  \leq\underset{%
\mathbb{R}
^{N}}{\int}1_{B_{\varepsilon}(\partial T_{\omega}U_{i})}(T_{\omega,i}%
y_{\omega,i})\underset{B_{s(\omega)\varepsilon}(y_{\omega,i})}{E\sup
}\left\vert f\right\vert dy_{\omega,i}\text{.}\nonumber
\end{align}
By Proposition \ref{properties} $($iv$)$, choosing $a=s(\omega)\varepsilon$,
$b=(S-s(\omega))\varepsilon_{0}$ and $c=S\varepsilon_{0}$, we get \eqref{aa}
is less than or equal to%
\[
\underset{%
\mathbb{R}
^{N}}{\int}\frac{1_{B_{\varepsilon}(\partial T_{\omega}U_{i})}(T_{\omega,i}%
y)}{m(B_{(S-s(\omega))\varepsilon_{0}}(y))}dy\underset{B_{(S-s(\omega
))\varepsilon_{0}}(y)}{\int}(\left\vert f\right\vert (z)+osc(f,B_{S\varepsilon
_{0}}(z)))dz\text{,}%
\]
which becomes, after changing the order of integration,
\[
\underset{%
\mathbb{R}
^{N}}{\int}(\left\vert f\right\vert (z)+osc(f,B_{S\varepsilon_{0}%
}(z)))dz\underset{%
\mathbb{R}
^{N}}{\int}\frac{1_{(T_{\omega,i})^{-1}B_{\varepsilon}(\partial T_{\omega
}U_{i})}(y)1_{B_{(S-s(\omega))\varepsilon_{0}}(z)}(y)}{m(B_{(S-s(\omega
))\varepsilon_{0}}(y))}dy\text{.}%
\]
Finally, since the measure of a ball depends only on its radius, we can
replace the second integral by%
\[
\frac{m(T_{\omega,i}{}^{-1}B_{\varepsilon}(\partial T_{\omega}U_{i})\cap
B_{(S-s(\omega))\varepsilon_{0}}(z))}{m(B_{(S-s(\omega))\varepsilon_{0}}%
(z))}\text{.}%
\]
By the definitions of $G_{\omega,\varepsilon_{0}}(\varepsilon)$ in \eqref{gg},
we get%
\begin{equation}
\frac{1}{(1+cs(\omega)^{\alpha}\varepsilon^{\alpha})}\underset{%
\mathbb{R}
^{N}}{\int}%
{\displaystyle\sum\limits_{i}}
R_{\omega,i}^{(2)}(x)dx\leq G_{\omega,\varepsilon_{0}}(\varepsilon
)(S\varepsilon_{0})^{\alpha}\left\vert f\right\vert _{\alpha}+G_{\omega
,\varepsilon_{0}}(\varepsilon)\left\Vert f\right\Vert _{L^{1}}\text{.}
\label{twoi}%
\end{equation}
By combining \eqref{kekeke} and \eqref{twoi} into \eqref{falafel} and dividing
both sides by $\varepsilon^{\alpha}$ and taking the supremum over all
$\varepsilon\leq\tilde{\varepsilon}_{0}=S\varepsilon_{0}$, we get%
\begin{equation}
\left\Vert \mathcal{L}_{\omega}f\right\Vert _{\alpha}\leq\eta(\omega
)\left\Vert f\right\Vert _{\alpha}+D(\omega)\left\Vert f\right\Vert _{L^{1}%
}\text{,} \label{LY}%
\end{equation}
where%
\begin{align}
\eta(\omega)  &  =(1+cs(\omega)^{\alpha}\varepsilon_{0}^{\alpha}%
)\Big(s(\omega)^{\alpha}+2\underset{\varepsilon\leq\varepsilon_{0}}{\sup}%
\frac{G_{\omega,\varepsilon_{0}}(\varepsilon)}{\varepsilon^{\alpha}%
}(S\varepsilon_{0})^{\alpha}\Big)\text{,}\label{eta}\\
&  =(1+cs(\omega)^{\alpha}\varepsilon_{0}^{\alpha})\zeta_{\varepsilon_{0}%
}(\omega)\text{,}\nonumber
\end{align}%
\begin{align}
D(\omega)  &  =cs(\omega)^{\alpha}+2(1+cs(\omega)^{\alpha}\varepsilon
_{0}^{\alpha})\underset{\varepsilon\leq\varepsilon_{0}}{\sup}\frac
{G_{\omega,\varepsilon_{0}}(\varepsilon)}{\varepsilon^{\alpha}}%
\label{dconstant}\\
&  \leq cs(\omega)^{\alpha}+(1+cs(\omega)^{\alpha}\varepsilon_{0}^{\alpha
})\zeta_{\varepsilon_{0}}(\omega)(\max(S,1)\varepsilon_{0})^{-\alpha}%
\text{.}\nonumber
\end{align}
Taking $\varepsilon_{0}$ small enough and using \eqref{iugb}, we have
\[
\int_{\Omega}\log(\eta(\omega))d\mathbb{P(\omega)}<0\text{.}%
\]
\end{proof}

The following lemma is taken from \cite{CQ} applied to our setting.

\begin{lemma}
[{\cite[Lemma C.$5$]{CQ}}]\label{ioc}Suppose we have the following inequality
\[
\left\Vert \mathcal{L}_{\omega}f\right\Vert _{\alpha}\leq A(\omega)\left\Vert
f\right\Vert _{\alpha}+B(\omega)\left\Vert f\right\Vert _{L^{1}}\text{,}%
\]
for all $f\in V_{\alpha}$ where $A(\omega)$ and $B(\omega)$ are measurable and
$\int_{\Omega}\log(A(\omega))d\mathbb{P(\omega)<}0$. Then there exists a full
measure subset $\Omega_{1}\subseteq\Omega$ with the following property
\[
\lim_{k\rightarrow\infty}\frac{1}{k}\log\Vert\mathcal{L}_{\bar{\omega}}%
^{(k)}\Vert_{ic(X)}\leq\int_{\Omega}\log(A(\omega))d\mathbb{P(\omega)}\text{
for all }\bar{\omega}\in\Omega_{1}\text{.}%
\]
\end{lemma}

\begin{corollary}
\label{quasicom}Let $T=\{T_{\omega}\}_{\omega\in\Omega}$ be a random Saussol
map. Provided $\varepsilon_{0}$ in \eqref{bdedosc} is small enough, the
following hold. (i) The random dynamical system generated by $T$ is
quasi-compact and (ii) its maximal Lyapunov exponent $\lambda^{\ast}$ is zero.
\begin{proof}
[Proof of Corollary \ref{quasicom} (i)]By Lemma \ref{ioc}\ and \eqref{LY}, the
index of compactness $\mathcal{K}^{\ast}$ is bounded above by $\int_{\Omega
}\log(\eta(\omega))d\mathbb{P(\omega)<}0$. Next, we show that $\lambda^{\ast}$
$\geq0$. Since the Perron--Frobenius operator $\mathcal{L}_{\omega}^{(k)}$ is
a Markov operator for each $\omega\in\Omega$, then for any density function
$f\in V_{\alpha}$, we have that
\[
\Vert\mathcal{L}_{\omega}^{(k)}f\Vert_{\alpha}\geq\Vert\mathcal{L}_{\omega
}^{(k)}f\Vert_{L^{1}}=\Vert f\Vert_{L^{1}}=1\text{,}%
\]
which shows that
\begin{equation}
\lambda^{\ast}\geq0. \label{haciiii}%
\end{equation}
\end{proof}
\begin{proof}
[Proof of Corollary \ref{quasicom} (ii)]To prove $\lambda^{\ast}\leq0$. We
have $\left\Vert \mathcal{L}_{\omega}\right\Vert _{1}\leq1$, it suffices to
consider the growth of the $\alpha$-oscillation of the term $\mathcal{L}%
_{\omega}^{(k)}f$. Applying the arguments in Lemma C.$5$ in \cite{CQ} and
Proposition $1.4$ in \cite{B}, the functions $\eta(\omega)$ and $D(\omega)$
can be redefined such that \eqref{LY} is satisfied and $D(\omega)$ is
uniformly bounded by $\tilde{D}$ and $\tilde{D}\geq\xi\tilde{\varepsilon}%
_{0}^{-\alpha}$, where $\xi$ is defined the proof of Theorem \ref{muer}.
Therefore, we have%
\begin{equation}
\left\Vert \mathcal{L}_{\omega}f\right\Vert _{\alpha}\leq\eta(\omega
)\left\Vert f\right\Vert _{\alpha}+\tilde{D}\left\Vert f\right\Vert _{L^{1}%
}\text{,} \label{HLY}%
\end{equation}
and $\int_{\Omega}\log(\eta(\omega))d\mathbb{P(\omega)<}0$. Iterating
\eqref{HLY}, we get a bound on the sequence $(|\mathcal{L}_{\omega}%
^{(k)}f|_{\alpha})_{k=1}^{\infty}$. Hence,
\[
\lim_{k\rightarrow\infty}\frac{1}{k}\log\Vert\mathcal{L}_{\omega}^{(k)}%
f\Vert_{\alpha}\leq0\text{.}%
\]
and since this is true for $\mathbb{P}$-a.e. $\omega\in\Omega$, we get
$\lambda^{\ast}\leq0$.\ By \eqref{haciiii}, we have $\lambda^{\ast}=0$.
\end{proof}
\end{corollary}

\section{Existence of random invariant and skew product ACIPs, finiteness, and
physicality of measures}

\begin{definition}
Let $T=\{T_{\omega}\}_{\omega\in\Omega}$ be a random Saussol map. A family
$\{\mu_{\omega}\}_{\omega\in\Omega}$ is called a random invariant measure for
$T$ if $\mu_{\omega}$ is a probability measure on $D$, the map $\omega
\mapsto\mu_{\omega}$ is measurable and%
\[
T_{\omega}\mu_{\omega}=\mu_{\sigma\omega}\text{, for }\mathbb{P}\text{-a.e.
}\omega\in\Omega\text{.}%
\]
A family $\{h_{\omega}\}_{\omega\in\Omega}$ is called a random invariant
density for $T$ if $h_{\omega}\geq0$, $h_{\omega}\in L^{1}(C)$, $\left\Vert
h_{\omega}\right\Vert _{L^{1}}=1$, the map $\omega\mapsto h_{\omega}$ is
measurable and%
\[
\mathcal{L}_{{\omega}}h_{\omega}=h_{\sigma\omega}\text{, for }\mathbb{P}%
\text{-a.e. }\omega\in\Omega\text{.}%
\]
\end{definition}

\begin{theorem}
\label{muer}Consider a random Saussol map $T$. If $\varepsilon_{0}$ is small
enough, then for each $\omega\in\Omega\ $and $k=1,2,\dots$, we define
\[
h_{\omega}^{k}=(\mathcal{L}_{\sigma^{-1}\omega}\circ\dots\circ\mathcal{L}%
_{\sigma^{-(k-1)}\omega}\circ\mathcal{L}_{\sigma^{-k}\omega})1\text{,}%
\]
where $1\in V_{\alpha}$ is the constant function and for each $s=1,2,\dots$,
we define%
\[
H_{\omega}^{s}=\frac{1}{s}\sum_{k=1}^{s}h_{\omega}^{k}\text{.}%
\]
Then, for $\mathbb{P}$-a.e. $\omega\in\Omega$:\newline(i) the sequence
$\{H_{\omega}^{s}\}_{s\in\mathbb{N}}$ is relatively compact in $L^{1}$;
and\newline(ii) the following limit exists,%
\begin{equation}
\lim_{s\rightarrow\infty}H_{\omega}^{s}=:h_{\omega}\in V_{\alpha}\text{ in
}L^{1}\text{.} \label{hw}%
\end{equation}
Moreover, $\{h_{\omega}\}_{\omega\in\Omega}$ is a random invariant density for
$T$.
\end{theorem}

\begin{proof}
For $k=1,2,\dots$, and $\mathbb{P}$-a.e. $\omega\in\Omega$, the following
holds,
\[
h_{\omega}^{k}=\mathcal{L}_{\sigma^{-k}\omega}^{(k)}1\text{.}%
\]
Applying the hybrid Lasota-Yorke inequality \eqref{HLY}\ to estimate $\Vert
h_{\omega}^{k}\Vert_{\alpha}$, we get%
\begin{align*}
&  \left\Vert h_{\omega}^{k}\right\Vert _{\alpha}=\left\Vert \mathcal{L}%
_{\sigma^{-k}\omega}^{(k)}1\right\Vert _{\alpha}=\left\Vert \mathcal{L}%
_{\sigma^{-1}\omega}((\mathcal{L}_{\sigma^{-2}\omega}\circ...\circ
\mathcal{L}_{\sigma^{-(k-1)}\omega}\circ\mathcal{L}_{\sigma^{-k}\omega
})1)\right\Vert _{\alpha}\\
&  \leq\eta(\sigma^{-1}\omega)\left\Vert \mathcal{L}_{\sigma^{-2}\omega
}(\mathcal{L}_{\sigma^{-3}\omega}\circ...\circ\mathcal{L}_{\sigma
^{-(k-1)}\omega}\circ\mathcal{L}_{\sigma^{-k}\omega})1\right\Vert _{\alpha}\\
&  +\tilde{D}\left\Vert (\mathcal{L}_{\sigma^{-2}\omega}\circ...\circ
\mathcal{L}_{\sigma^{-(k-1)}\omega}\circ\mathcal{L}_{\sigma^{-k}\omega
})1\right\Vert _{L^{1}}\\
&  \leq\eta(\sigma^{-1}\omega)\eta(\sigma^{-2}\omega)...\eta(\sigma^{-k}%
\omega)\left\Vert 1\right\Vert _{\alpha}\\
&  +\tilde{D}\left\Vert (\mathcal{L}_{\sigma^{-2}\omega}\circ...\circ
\mathcal{L}_{\sigma^{-(k-1)}\omega}\circ\mathcal{L}_{\sigma^{-k}\omega
})1\right\Vert _{L^{1}}\\
&  +\tilde{D}\eta(\sigma^{-1}\omega)\left\Vert (\mathcal{L}_{\sigma^{-3}%
\omega}\circ...\circ\mathcal{L}_{\sigma^{-(k-1)}\omega}\circ\mathcal{L}%
_{\sigma^{-k}\omega})1\right\Vert _{L^{1}}\\
&  +\tilde{D}\eta(\sigma^{-2}\omega)\eta(\sigma^{-1}\omega)\left\Vert
(\mathcal{L}_{\sigma^{-4}\omega}\circ...\circ\mathcal{L}_{\sigma
^{-(k-1)}\omega}\circ\mathcal{L}_{\sigma^{-k}\omega})1\right\Vert _{L^{1}}\\
&  +...+\tilde{D}\eta(\sigma^{-k}\omega)...\eta(\sigma^{-2}\omega)\eta
(\sigma^{-1}\omega)\left\Vert 1\right\Vert _{L^{1}}\text{.}%
\end{align*}
Since$\left\Vert 1\right\Vert _{\alpha}=0$, $\left\Vert 1\right\Vert _{L^{1}%
}=m(C)$ and the transfer operator is contractive in $L^{1}$, we have%
\begin{align}
\left\Vert h_{\omega}^{k}\right\Vert _{\alpha}  &  \leq\tilde{D}%
m(C)\Big(1+\eta(\sigma^{-1}\omega)+\eta(\sigma^{-2}\omega)\eta(\sigma
^{-1}\omega)+...+\eta(\sigma^{-k}\omega)...\eta(\sigma^{-1}\omega
)\Big)\nonumber\\
&  =\tilde{D}m(C)\Big(1+\sum_{j=1}^{k}\eta^{(j)}(\sigma^{-j}\omega
)\Big)\text{,} \label{fac}%
\end{align}
where for each $j\in%
\mathbb{N}
$, $\eta^{(j)}(\sigma^{-j}\omega):=\eta(\sigma^{-1}\omega)\eta(\sigma
^{-2}\omega)...\eta(\sigma^{-j}\omega)$. Note that for $j=1,2,...$, we have%
\[
\frac{1}{j}\log\eta^{(j)}(\sigma^{-j}\omega)=\frac{1}{j}\sum_{t=1}^{j}\log
\eta(\sigma^{-t}\omega)\text{.}%
\]
Applying Birkhoff ergodic theorem, we get that the above time average
converges as $j\rightarrow\infty$ to the space average $\int_{\Omega}\log
(\eta(\omega))d\mathbb{P(\omega)}=:\log(\hat{\xi})<0$, for some $0<\hat{\xi
}<1$. Choose $\xi$ such that $0<\hat{\xi}<\xi<1$ and $\xi>\frac{1}{2}$. For
large enough $j_{0}(\omega)$, we have that%
\[
\eta^{(j)}(\sigma^{-j}\omega)<\xi^{j}\text{, for all }j\geq j_{0}%
(\omega)\text{.}%
\]
Let $\theta(\omega)$ be defined as%
\begin{equation}
\theta(\omega):=\underset{1\leq j\leq j_{0}(\omega)}{\max}(\frac{\eta
^{(j)}(\sigma^{-j}\omega)}{\xi^{j}},1)\text{,} \label{wdjqiwh}%
\end{equation}
and hence for all $j$, we have that%
\[
\eta^{(j)}(\omega)<\theta(\omega)\xi^{j}\text{.}%
\]
Taking the sum over $j$, by \eqref{fac}, we get%
\begin{align}
\left\Vert h_{\omega}^{k}\right\Vert _{\alpha}  &  \leq\tilde{D}%
m(C)(1+\theta(\omega)\sum_{j=1}^{k}\xi^{j})\label{fyu}\\
&  \leq\tilde{D}m(C)(1+\theta(\omega)\sum_{j=1}^{\infty}\xi^{j})=\tilde
{D}m(C)(1+\frac{\theta(\omega)}{1-\xi})\text{,}\nonumber
\end{align}
since $0<\xi<1$. Let
\begin{equation}
\Theta(\omega):=\tilde{D}m(C)(1+\frac{\theta(\omega)}{1-\xi})\text{,}
\label{bdonhomega}%
\end{equation}
then we have proven that for every $k\in\mathbb{N}$
\begin{equation}
\left\Vert h_{\omega}^{k}\right\Vert _{\alpha}\leq\Theta(\omega)\text{.}
\label{laii}%
\end{equation}
From this inequality, it follows $\{\left\Vert h_{\omega}^{k}\right\Vert
_{\alpha}\}_{k\in\mathbb{N}}$ is bounded and hence the sequence of averages
$\{\left\Vert H_{\omega}^{s}\right\Vert _{\alpha}\}_{s\in\mathbb{N}}$ too.
Therefore, $\{H_{\omega}^{s}\}_{s\in\mathbb{N}}$ is relatively compact in
$L^{1}$ by Lemma A.$1$\ in \cite{liv}. This establishes (i). Then, the same
argument in the proof of \cite[Theorem $4.2$]{BGT}, we have $\{H_{\omega}%
^{s}\}_{s\in\mathbb{N}}$ converges in the strong sense to a random invariant
density $h_{\omega}$, as in \eqref{hw}. The relative compactness of
$V_{\alpha}$ in $L^{1}$ implies that $h_{\omega}\in V_{\alpha}$. This proves (ii).
\end{proof}

\begin{remark}
\label{acip} For $\mathbb{P}$-a.e. $\omega\in\Omega$, Let $\mu_{\omega}$ on
the fiber $\{\omega\}\times C\subset\Omega\times C$, as
\begin{equation}
\frac{d\mu_{\omega}}{dm}=h_{\omega}\text{,} \label{samacip}%
\end{equation}
where $h_{\omega}$ is given by \eqref{hw}. Then, $\mu_{\omega}$ is a random
invariant ACIP and the measure $\mu$\ defined on $\mathbb{P\times}%
m$-measurable sets $A\subseteq$ $\Omega\times C$ by
\begin{equation}
\mu(A)=\int\limits_{\Omega}\mu_{\omega}(A)d\mathbb{P(\omega)}\text{,}
\label{skewacip}%
\end{equation}
is an ACIP for the associated skew product $F$ defined in \eqref{sp}.
\end{remark}

For the rest of the paper, we assume%
\[
\int_{\Omega}\log^{+}\left\Vert \mathcal{L}_{\omega}\right\Vert _{\alpha
}d\mathbb{P(\omega)<}\infty\text{.}%
\]

By Corollary~\ref{quasicom}, random Saussol maps give rise to quasi-compact
random dynamical systems with $\lambda_{1}=0$. Therefore, Theorem $17$ in
\cite{FLQ} implies the following.

\begin{corollary}
\label{coro1} For $\mathbb{P}$-a.e. $\omega\in\Omega$, the random invariant
density $h_{\omega}$ given in \eqref{hw} belongs to the Oseledets space
$Y_{1}(\omega)$ given in \eqref{y1}. Moreover, the number $r$ of ergodic skew
product ACIPs $\mu_{1},...,\mu_{r}$ defined in Equation \eqref{skewacip} is
finite and
\begin{equation}
r\leq d_{1}=\dim(Y_{1}(\omega))\text{.} \label{easybd}%
\end{equation}
\end{corollary}

The proof of this corollary is the same as the proof of Corollary $4.6$ in
\cite{BGT}.

\begin{definition}
\label{phme} Consider the tuple $(\Omega,\mathcal{F},\mathbb{P},\sigma,T)$
where $(\Omega,\mathcal{F},\mathbb{P)}$ is a probability space, $\sigma
:\Omega\circlearrowleft$ an ergodic and invertible $\mathbb{P-}$preserving
transformation and $T=\{T_{\omega}:C\circlearrowleft\}_{\omega\in\Omega}$
where $C\subseteq%
\mathbb{R}
^{n}$. A probability measure $\nu$ on $C$ is called \emph{physical }if for
$\mathbb{P}$-a.e. $\omega\in\Omega$, the Lebesgue measure of the random basin
$RB_{\omega}(\nu)$ of $\nu$ at $\omega$ is positive where%
\begin{equation}
RB_{\omega}(\nu)=\{x\in C:\text{ }\frac{1}{s}\sum_{k=0}^{s-1}\delta
_{T_{\omega}^{(k)}(x)}\rightarrow\nu\text{ as }s\rightarrow\infty\}\text{,}
\label{hhhoh}%
\end{equation}
where $\delta_{x}$ is the Dirac measure at a point $x$ and the convergence in
\eqref{hhhoh} is in the weak convergence sense.
\end{definition}

The next result due to Buzzi applies in our setting.

\begin{theorem}
[{\cite[Proposition $4.1$]{B}}]\label{ph} Let $\mu$ be one of the measures
$\mu_{i}:$ $i=1,...r$ given in Corollary \ref{coro1}. Then, the marginal
measure of $\mu$ on $C$, denoted by $\nu$, is a physical measure on $C$.
\end{theorem}

The union of all basins of the of the physical measures $\nu_{i}$ coming from
the marginals of $\mu_{i}$ on $C$, $i=1,...r$\ has full Lebesgue measure,
which means Lebesgue almost everywhere, the asymptotic long term behaviour of
a full $\mathbb{P}$-measure set of random orbits will be described by these
physical measures.

\section{An upper bound on the number of ergodic skew product ACIPs}

\begin{theorem}
\label{ph copy(2)} Assume that $m(C)\geq1.$ The number $r$ of ergodic skew
product ACIPs defined in Corollary \ref{coro1} satisfies
\begin{equation}
r\leq\frac{m(C)}{\gamma_{N}}\underset{\omega\in\Omega}{E\inf}\left(
\Theta(\omega)^{\frac{N}{\alpha}}\right)  \text{,} \label{skewbd}%
\end{equation}
where $\Theta(\omega)$ is defined in \eqref{bdonhomega} and $\gamma_{N}$ is
the volume of the $N$-dimensional unit ball.
\end{theorem}

\begin{proof}
For $\mathbb{P}$-a.e. $\omega\in\Omega$, by \eqref{laii}, we know that
$\left\Vert h_{\omega}\right\Vert _{\alpha}\leq\Theta(\omega)$ and $\int
h_{\omega}dm=1$. By Lemma \ref{ph copy(1)}, for $\mathbb{P}$-a.e. $\omega
\in\Omega$, the infimum of $h_{\omega}$ is strictly positive on some ball of
radius
\begin{equation}
\min(\tilde{\varepsilon}_{0},\left(  \frac{1}{\Theta(\omega)}\right)
^{\frac{1}{\alpha}})=\left(  \frac{1}{\Theta(\omega)}\right)  ^{\frac
{1}{\alpha}}. \label{ubgb}%
\end{equation}
This last step is justified as follows. By \eqref{bdonhomega}, we have
\[
\Theta(\omega)=\tilde{D}m(C)(1+\frac{\theta(\omega)}{1-\xi})\text{.}%
\]
By \eqref{wdjqiwh}, we have $\theta(\omega)\geq1$ and $m(C)\geq1$ by
assumption, therefore%
\[
\Theta(\omega)\geq\frac{\tilde{D}}{1-\xi}\text{.}%
\]
In the proof of Corollary \ref{quasicom}, $\tilde{D}$ is chosen such that
$\tilde{D}\geq\xi\tilde{\varepsilon}_{0}^{-\alpha}$, thus%
\[
\Theta(\omega)\geq\frac{\xi}{1-\xi}\tilde{\varepsilon}_{0}^{-\alpha}\text{.}%
\]
Note that $\xi$, introduced in the proof of Theorem \ref{muer}, can be chosen
such that $\xi>\frac{1}{2}$ (since otherwise, we can choose $\xi$ to be
$1-\xi$), thus
\[
\left(  \frac{1}{\Theta(\omega)}\right)  ^{\frac{1}{\alpha}}<\tilde
{\varepsilon}_{0}\text{.}%
\]
Since \eqref{ubgb} is true for $\mathbb{P}$-a.e. $\omega\in\Omega$, it follows
that the number $r$ of ergodic skew product ACIPs is bounded above by the
essential infimum of maximal number of balls of radius $\left(  \frac
{1}{\Theta(\omega)}\right)  ^{\frac{1}{\alpha}}$ contained in $C$, which is
bounded by
\[
\frac{m(C)}{\gamma_{N}}\underset{\omega\in\Omega}{E\inf}\left(  \Theta
(\omega)^{\frac{N}{\alpha}}\right)  .
\]
\end{proof}

\begin{remark}
If the upper bound in \eqref{skewbd} is strictly less than $2$, then we get
uniqueness of the number of ergodic skew product ACIPs. In such a case, one
can use the results given in \cite{Daaa,aiuwo} to investigate quenched limit
theorems in this setting.
\end{remark}

The next example shows how to verify \textbf{(}PE5\textbf{) }once the
partition is finite and the boundaries of the $U_{i}$'s are piecewise smooth
boundaries. This example is motivated from Lemma $2.1$\ in \cite{S} adapted to
our random setting.

\begin{example}
\label{le copy(3)} In Definition \ref{le copy(1)}, suppose that $T$ satisfies
\textbf{(}PE1\textbf{) }through\textbf{ }(PE4\textbf{) }and $\mathcal{I}$ is
finite such that the boundaries of the $U_{\omega,i}$'s are included in
piecewise $C^{1}$ codimension one embedded compact submanifolds. Denote by
\[
Y(\omega):=\underset{x\in%
\mathbb{R}
^{N}}{\sup}\text{ }%
{\displaystyle\sum\limits_{i\in\mathcal{I}}}
\#\{\text{smooth pieces intersecting }\partial U_{\omega,i}\text{ containing
}x\}
\]
and
\begin{equation}
\Lambda(\omega):=s(\omega)^{\alpha}+4Y(\omega)\frac{\gamma_{N-1}}{\gamma_{N}%
}\frac{s(\omega)}{(S-s(\omega))}S^{\alpha}\text{.} \label{icbq}%
\end{equation}
Suppose that there exists $\varrho\mathbb{<}0$ such that $\int_{\Omega}%
\log(\Lambda(\omega))d\mathbb{P(\omega)<}\varrho$, then \textbf{(}PE5\textbf{)
}holds. To see this, fix $\omega\in\Omega$, $i\in\mathcal{I}$, $\varepsilon
\leq\varepsilon_{0}$ and $x\in%
\mathbb{R}
^{N}$. By \textbf{(}PE1\textbf{)}, we have%
\[
T_{\omega,i}^{-1}B_{\varepsilon}(\partial T_{\omega}U_{\omega,i})\cap
B_{(S-s(\omega))\varepsilon_{0}}(x)\subset B_{s(\omega)\varepsilon}(\partial
U_{\omega,i})\cap B_{(S-s(\omega))\varepsilon_{0}}(x)\text{.}%
\]
By the assumption, we have
\[
\partial U_{\omega,i}=%
{\displaystyle\bigcup\limits_{j\in\mathcal{J}_{\omega,i}}}
\Gamma_{\omega,i,j}\text{,}%
\]
where $\mathcal{J}_{\omega,i}$ is a finite indexing set and $\Gamma
_{\omega,i,j}$ is a compact $C^{1}$ embedded submanifold. Therefore, we have
\[
T_{\omega,i}^{-1}B_{\varepsilon}(\partial T_{\omega}U_{\omega,i})\cap
B_{(S-s(\omega))\varepsilon_{0}}(x)\subset%
{\displaystyle\bigcup\limits_{j\in\mathcal{J}_{\omega,i}}}
B_{s(\omega)\varepsilon}(\Gamma_{\omega,i,j})\cap B_{(S-s(\omega
))\varepsilon_{0}}(x)\text{.}%
\]
Arguing as in the proof of Lemma $2.1$ in \cite{S}, for small $\varepsilon$,
we get%
\[
m\Big(B_{s(\omega)\varepsilon}(\Gamma_{\omega,i,j})\cap B_{(S-s(\omega
))\varepsilon_{0}}(x)\Big)\leq2s(\omega)\varepsilon\gamma_{N-1}((S-s(\omega
))\varepsilon_{0})^{N-1}(1+o(1))\text{,}%
\]
which implies%
\begin{equation}
G_{\omega,\varepsilon_{0}}(\varepsilon)\leq2Y(\omega)\frac{\gamma_{N-1}%
}{\gamma_{N}}\frac{s(\omega)\varepsilon}{(S-s(\omega))\varepsilon_{0}%
}(1+o(1))\text{.} \label{hgfrd}%
\end{equation}
Since the number of the $\Gamma_{\omega,i,j}$'s is finite, by taking
$\varepsilon_{0}$ small enough, we get%
\begin{align*}
\int_{\Omega}\log(\zeta_{\varepsilon_{0}}(\omega))d\mathbb{P(\omega)}  &
\mathbb{=}\int_{\Omega}\log(s(\omega)^{\alpha}+2\underset{\varepsilon
\leq\varepsilon_{0}}{\sup}\frac{G_{\omega,\varepsilon_{0}}(\varepsilon
)}{\varepsilon^{\alpha}}(S\varepsilon_{0})^{\alpha})d\mathbb{P(\omega)}\\
&  \mathbb{\leq}\int_{\Omega}\log(s(\omega)^{\alpha}+4Y(\omega)\frac
{\gamma_{N-1}}{\gamma_{N}}\frac{s(\omega)}{(S-s(\omega))}S^{\alpha
})d\mathbb{P(\omega)}\\
&  \mathbb{=}\int_{\Omega}\log(\Lambda(\omega))d\mathbb{P(\omega)<}%
\varrho\text{.}%
\end{align*}
\end{example}

\section*{Acknowledgments}

The authors have been partially supported by the Australian Research Council
(DE160100147). The authors are thankful to Prof. Beno\^{\i}t Saussol for
useful discussions, and to a referee for valuable input. Fawwaz Batayneh
acknowledges the support of the University of Queensland through an Australian
Government Research Training Program Scholarship.

\end{document}